\documentclass[12pt]{amsart}
\usepackage{latexsym}
\usepackage{amsmath}
\usepackage{amssymb}
\usepackage{amscd}
\usepackage{amsthm}
\usepackage{xypic}
\xyoption{curve}
\usepackage{ifthen} 
\usepackage{hyperref} 
\usepackage{graphicx}
\theoremstyle{plain}
 \newtheorem{theorem}{Theorem}[section]
 \newtheorem{proposition}[theorem]{Proposition}
 \newtheorem{lemma}[theorem]{Lemma}
 \newtheorem{corollary}[theorem]{Corollary}
 
\theoremstyle{definition}
 \newtheorem{definition}[theorem]{Definition}

\theoremstyle{remark}

\makeindex
\begin{document}
\title[Stable cohomology]{Stable cohomology of alternating groups}

\author[Bogomolov]{Fedor Bogomolov$^1$}
\address{F. Bogomolov, Courant Institute of Mathematical Sciences\\
251 Mercer St.\\
New York, NY 10012, U.S.A., \emph{and}}
\address{Laboratory of Algebraic Geometry, GU-HSE\\
7 Vavilova Str.\\
Moscow, Russia, 117312}
\email{bogomolo@courant.nyu.edu}
\author[B\"ohning]{Christian B\"ohning$^2$}
\address{Christian B\"ohning, Fachbereich Mathematik der Universit\"at Hamburg\\
Bundesstra\ss e 55\\
20146 Hamburg, Germany}
\email{christian.boehning@math.uni-hamburg.de}

\thanks{$^1$ Supported by NSF grant DMS-1001662 and by AG Laboratory GU- HSE grant RF government ag. 11 11.G34.31.0023}
\thanks{$^2$ Supported by the German Research Foundation (Deutsche Forschungsgemeinschaft) through Heisenberg-Stipendium BO 3699/1-1}

\newcommand{\PP}{\mathbb{P}} 
\newcommand{\QQ}{\mathbb{Q}} 
\newcommand{\ZZ}{\mathbb{Z}} 
\newcommand{\CC}{\mathbb{C}} 
\newcommand{\rmprec}{\wp}
\newcommand{\rmconst}{\mathrm{const}}
\newcommand{\xycenter}[1]{\begin{center}\mbox{\xymatrix{#1}}\end{center}} 
\newboolean{xlabels} 
\newcommand{\xlabel}[1]{ 
                        \label{#1} 
                        \ifthenelse{\boolean{xlabels}} 
                                   {\marginpar[\hfill{\tiny #1}]{{\tiny #1}}} 
                                   {} 
                       } 
\setboolean{xlabels}{false} 

\

\begin{abstract}
In this article we determine the stable cohomology groups $H^i_{\mathrm{s}} (\mathfrak{A}_n, \: \ZZ /p \ZZ )$ of the alternating groups $\mathfrak{A}_n$ for all integers $n$ and $i$, and all primes $p$. 
\end{abstract}

\maketitle

\section{Introduction and preliminaries}\xlabel{sIntroduction}
Let $G$ be a finite group, $V$ a finite-dimensional generically free complex representation of $G$, and let $V^L \subset V$ be the nonempty open subset of $V$ on which the $G$-action is free. There is then a natural homotopy class of maps from the classifying space $\mathrm{B} G$ to $V^L/G$ which, for each nonempty $G$-invariant Zariski open subset $U\subset V^L$ gives maps $H^i (G, \: \ZZ/p\ZZ ) \to H^i (U/G , \ZZ / p \ZZ )$. It turns out that the kernel $K_{G, \: V}$ of
\[
H^i (G, \: \ZZ /p\ZZ ) \to  \varinjlim_U H^i (U/G , \: \ZZ /p\ZZ )
\]
is independent of $V$, and the stable cohomology $H^i_{\mathrm{s}} (G, \: \ZZ/p\ZZ )$ is defined to be the quotient $H^i (G, \: \ZZ/p\ZZ )/K_{V, \: G}$. Let $K= \CC (V)^G$.  Algebraically, $H^i_{\mathrm{s}} (G, \: \ZZ/p\ZZ )$ can be identified with the image of $H^i (G, \: \ZZ / p\ZZ )$ in $H^i (\mathrm{Gal}(K) , \: \ZZ / p\ZZ )$.

\

More accessible computationally and stable birational invariants of the function field $K$ are the unramified cohomology groups $\mathrm{H}^i_{\mathrm{nr}} (G, \: \ZZ /p\ZZ )$ defined as follows: geometrically, the unramified cohomology classes $a$ inside $\mathrm{H}^i_{\mathrm{s}} (G, \: \ZZ /p \ZZ )$ are those for which, given any divisorial valuation $\nu_D$ of $K$, there exists a normal model $X=X_D$ of $K$ on which $\nu_D$ has a center, an isomorphism $i \, : \, U_X \to U_{V^L/G}$ between nonempty open subsets $U_X$ of $X$ and $U_{V^L/G}$ of $V^L/G$, and a representative $a'$ of $a$ in $H^i (U_{V^L/G} , \: \ZZ /p\ZZ )$ such that there is a class $b\in H^i (X , \: \ZZ /p\ZZ )$ whose image in $H^i (U_X, \: \ZZ /p\ZZ )$ coincides with $i^{\ast } (a')$. More algebraically, if $\mathcal{O}_{\nu }\subset K$ is the valuation ring of $\nu$, $\kappa_{\nu }=\mathcal{O}_{\nu }/\mathfrak{m}_{\nu }$ its residue field,   $S=\mathrm{Spec} (\mathcal{O}_{\nu })$ with open subset the generic point $U = \mathrm{Spec} (K) \subset S$ and complement the closed point $Z=\mathrm{Spec} (\kappa_{\nu } ) \subset S$, one can write down the long exact sequence of \'{e}tale cohomology with supports 
\[
\dots \to H^i (S, \: \ZZ/p ) \to  H^i (U, \: \ZZ/p ) \to H^{i+1}_Z (S, \: \ZZ/p ) \to H^{i+1} (S, \: \ZZ/p ) \to \dots 
\] 
where $H^i (U, \: \ZZ/p) \simeq H^i (\mathrm{Gal } (K), \: \ZZ/p)$ and there is the cohomological purity isomorphism
\[
H^j_Z (S, \: \ZZ/p ) \simeq H^{j-2} (Z, \: \ZZ/p )
\]
whence the preceding sequence becomes the \emph{Gysin sequence}
\begin{gather*}
\dots \to H^i_{\mathrm{\acute{e}t}} (S, \: \ZZ/p ) \stackrel{r_{\nu }}{\to }  H^i (\mathrm{Gal}(K), \: \ZZ/p ) \\ \stackrel{\partial_{\nu }}{\to } H^{i-1} (\mathrm{Gal}(\kappa_{\nu}), \: \ZZ/p ) \to H^{i+1}_{\mathrm{\acute{e}t}} (S, \: \ZZ/p ) \to \dots \, .
\end{gather*}
A class in $H^i (\mathrm{Gal}(K), \: \ZZ/p )$ is clearly unramified according to the geometric definition if and only if it is in the image of all maps $r_{\nu }$ for $\nu$ running over the divisorial valuations of $K$, i.e. equivalently if it is in the kernel of all maps $ \partial_{\nu }$, the \emph{residue maps}. The preceding sequence has as topological analogue the Borel-Moore long exact sequence. The residue map
\[
\partial_{\nu }\, :\, H^i (\mathrm{Gal}(K), \: \ZZ/p ) \to H^{i-1} (\mathrm{Gal}(\kappa_{\nu}), \: \ZZ/p )
\] 
agrees -up to a sign- with the following map defined entirely within the framework of Galois cohomology (see e.g. \cite{GMS}, Chap. II of Serre's part, \S 6 and \S 7): extend $\nu$ in some way to a valuation $\nu^{\ast}$ on $\bar{K}$ which is possible by Chevalley's theorem; all such extensions are conjugate under $\Gamma =\mathrm{Gal}(K )$, and $\nu^{\ast }$ defines subgroups $\Gamma_Z\subset \Gamma$ (the decomposition group, \emph{Zerlegungsgruppe}) and $\Gamma_T\subset \Gamma$ (the inertia subgroup, \emph{Tr\"agheitsgruppe}) by the conditions that $\sigma\in\Gamma$ is in $\Gamma_Z$ if $\sigma\cdot \nu^{\ast}$ and $\nu^ {\ast }$ are equivalent valuations, i.e. have the same valuation ring, and $\Gamma_T$ consists of those $\sigma$ such that $\sigma\cdot x -x \in \mathfrak{M}_{\nu^{\ast }}$ for all $x$ in the valuation ring of $\nu^ {\ast }$ whose maximal ideal we denoted by $\mathfrak{M}_{\nu^{\ast }}$. The decomposition group can be identified with the Galois group $\mathrm{Gal} (\bar{K}_{\nu} /K_{\nu })$ of the completion of $K$ with respect to $\nu$. The residue map $\partial_{\nu }$ then factors over the restriction to the decomposition group
\[
\partial_{\nu } \, : \, H^i (\mathrm{Gal}(K), \: \ZZ/p ) \to H^i (\mathrm{Gal}(K_{\nu }), \: \ZZ/p)\stackrel{r}{\to }  H^{i-1} (\mathrm{Gal}(\kappa_{\nu}) ,\: \ZZ/p )
\]
where the second arrow has the following description in the local situation: the Galois group $\Gamma_{K_\nu }= \mathrm{Gal}( K_{\nu })$ sits in the exact sequence
\[
1 \to I \to \Gamma_{K_{\nu }} \to \Gamma_{\kappa_{\nu }} = \mathrm{Gal} (\kappa_{\nu })\to 1
\]
where the surjection is given by the fact that $\nu$ extends uniquely to $\bar{K}_{\nu}$ and the residue field of the extension is an algebraic closure of $\kappa_{\nu }$. The kernel is the inertia subgroup which we denote by $I$ in this context, and it is topologically cyclic, $I \simeq \hat{\ZZ}$, corresponding to taking roots of the uniformizing parameter, and the preceding sequence splits, $\Gamma_{K_{\nu }}\simeq \hat{\ZZ}\oplus \mathrm{Gal} (\kappa_{\nu })$. As $\hat{\ZZ}$ has cohomological dimension $1$, one gets $H^i (\Gamma_{K_{\nu}} , \: \ZZ/p ) \simeq H^i (\Gamma_{\kappa_{\nu }}, \: \ZZ /p\ZZ )\oplus H^{i-1}  (\Gamma_{\kappa_{\nu }}, \: \ZZ /p\ZZ ) $ and a projection, which is independent of the splitting, $H^i (\mathrm{Gal}(K_{\nu }), \: \ZZ/p)\stackrel{r}{\to }  H^{i-1} (\mathrm{Gal}(\kappa_{\nu}) ,\: \ZZ/p )$, defining the second arrow in the sequence of maps yielding $\partial_{\nu }$. More precisely, the Hochschild-Serre spectral sequence of the group extension of $\Gamma_{\kappa_{\nu}}$ by $I$
\[
H^p (\Gamma_{\kappa_{\nu }}, \: H^q (I, \: \ZZ/p) ) \implies H (\Gamma_{K_{\nu }}, \: \ZZ/p )
\]
reduces to a long exact sequence as $H^i (I, \: \ZZ/p ) = 0$ for $i \ge 2$, $H^0 (I, \: \ZZ/p ) = \ZZ/p$, $H^1 (I, \: \ZZ/p) = \mathrm{Hom} (I, \: \ZZ/p ) = \ZZ/p$, which reads
\begin{gather*}
\dots \to H^i (\Gamma_{\kappa_{\nu }}, \: \ZZ/p ) \to H^i (\Gamma_{K_{\nu }}, \: \ZZ/p ) \to H^{i-1} (\Gamma_{\kappa_{\nu }} , \: \mathrm{Hom} (\hat{\ZZÊ}, \: \ZZ/p )) \to \\
\to H^{i+1} (\Gamma_{\kappa_{\nu }}, \: \ZZ/p ) \to H^{i+1} (\Gamma_{K_{\nu }}, \: \ZZ/p ) \to \dots
\end{gather*}
and the fact that the extension splits implies that this long exact sequence breaks into short exact sequences
\[
0 \to H^i (\Gamma_{\kappa_{\nu }}, \: \ZZ/p ) \to H^i (\Gamma_{K_{\nu }}, \: \ZZ/p ) \stackrel{r}{\to } H^{i-1} (\Gamma_{\kappa_{\nu }} , \: \ZZ/p )
\]
where $r$ can also be explicitly described in terms of cocycles, see\cite{GMS}, p.16.

\

In \cite{B-P} the following theorem was proven (loc. cit, Theorem 5.1):
\begin{theorem}\xlabel{t2Coefficients}
There is a natural isomorphism $H^{\ast }_{\mathrm{s}} (\mathfrak{A}_{2n+1}, \: \ZZ/2\ZZ ) \simeq H^{\ast }_{\mathrm{s}} (\mathfrak{A}_{2n}, \: \ZZ/2\ZZ )$, and as a $\ZZ/2\ZZ$-vector space
\[
H^{\ast}_{\mathrm{s}} (\mathfrak{A}_{2n}, \: \ZZ/2\ZZ ) = \bigoplus_{0\le i \le n } \ZZ/2\ZZ \cdot w_{2i} \oplus \bigoplus_{0< i \le n}\ZZ/2\ZZ\cdot u_1 \wedge w_{2i}
\]
where $w_j$ are the (images in stable cohomology of) the Stiefel-Whitney classes in $H^j (\mathfrak{A}_{2n}, \: \ZZ/2\ZZ )$ obtained from the cohomology ring of the real orthogonal group $\mathrm{O}(2n)$ via the inclusions $\mathfrak{A}_{2n}\subset \mathfrak{S}_{2n}\subset \mathrm{O}(2n)$. The class $u_1$ is a one-dimensional cohomology class which can be described as follows:\\
Putting $N=2n$, the group $\mathfrak{S}_{N}$ acts generically freely on the complement $\CC^{N-1}-H$ of the braid hyperplane arrangement $H$ in the standard permutation representation $\CC^{N-1}$, and $(\CC^{N-1}-H)/\mathfrak{S}_N \simeq \CC^{N-1} -\Delta$, the complement of the discriminant. Taking a nonramified double covering $\widetilde{\CC^{N-1}-\Delta}$ of $\CC^{N-1} -\Delta$ corresponding to the inclusion $\mathfrak{A}_N \subset \mathfrak{S}_N$, one gets a description of $u_1$ as the generator of $H^1 (\widetilde{\CC^{N-1}-\Delta})$ given by the root of the discriminant.

\end{theorem}

In our Theorem \ref{tStableCohomologyAlternating} we determine $H^{\ast}_{\mathrm{s}} (\mathfrak{A}_{N}, \: \ZZ/p\ZZ )$ completely for odd primes $p$.

\

We base our approach to the computation of the stable cohomology of alternating groups on the following lemmas.

\begin{lemma}\xlabel{lKuennethFactorAbelian}
Suppose a group is a product $G\times A$ of finite groups $G$ and $A$ with $A$ abelian. Then there is the K\"unneth decomposition
\begin{gather*}
H^{n}_{\mathrm{s}} (G\times A, \: \ZZ/p ) \simeq \bigoplus_{i+j= n} H^i_{\mathrm{s}} (G, \: \ZZ/p ) \otimes H^j_{\mathrm{s}} (A, \: \ZZ/p ) \, .
\end{gather*}
\end{lemma}

\begin{proof}
It is known ( \cite{Bogo93}, Lemma 7.1) that if we choose a free presentation $\pi \, :\, \ZZ^n \twoheadrightarrow A$ of $A$, then the kernel of the stabilization map coincides with the kernel of $\pi^{\ast }$. In other words, if one realizes $\ZZ^n$ as the fundamental group of some algebraic torus $T$, with cover $T' \to T\simeq T' /A$ corresponding to $A$, and realizes $T'$ as a maximal torus in some $\mathrm{GL} (W)$, then stabilization is achieved by considering the image of the cohomology of $A$ in the cohomology of $T\simeq T'/A \subset W^L/A$. This is so because one can find a product of circles $(S^1)^m$ in the complement of any divisor $D$ in $T \simeq (\CC^{\ast })^m$, the inclusion being a homotopy equivalence, so the cohomology of $T$ is already stable. The product of circles can be found by induction on the dimension $m$ of $T$; if $m=1$, one chooses a circle in the complex plane $\CC$ not passing through the finite number of points which $D$ consists of. If $m>1$, one views $T$ as a subset of $W$, which is stratified into torus orbits of lower dimension. Each of these is isomorphic to an algebraic tori. Choose a codimension $1$ torus orbit $T_1$ adjacent to $T$. The closure $\bar{D}$ inside $W$ of a divisor $D\subset T$ meets $T_1$ in a proper algebraic subset, and by the induction hypothesis there is a real submanifold $M\simeq (S^1)^{m-1}$ in the complement of $\bar{D} \cap T_1$. If $x_1, \dots , \: x_m$ are coordinates in $W$ such that $T = \{ x_i \neq 0 \; \forall \, i \}$, $T_1 = \{ x_1 = 0 \wedge x_j \neq 0 \; \forall\, j \neq 1 \}$, then $W = W' \oplus \CC$ where $W' = \{ x_1 = 0 \}$. If we choose a small circle $S_{\epsilon }\subset \CC$, then $M \times S_{\epsilon } \subset W$ will be in a small neighbourhood of $M$ hence will not intersect $D$.

 \
 
This argument can be made relative: note first that there is always a natural surjection $H^{\ast }_{\mathrm{s}} (G, \: \ZZ/p ) \otimes H^{\ast }_{\mathrm{s}} (A, \: \ZZ/p ) \to H^{\ast }_{\mathrm{s}} (G \times A, \: \ZZ/p )$ as the Zariski topology on a product is finer than the product topology, and to show it is an isomorphism, it suffices to note the following: suppose $T\simeq (\CC^{\ast })^m \simeq T'/A$ is as before, and $V$ is a generically free $G$-representation, $V^L$ the open part where the action is free. Then if $D\subset (V^L/G) \times T$ is any divisor, there is always a divisor $D' \subset V^L/G$ and a relatively compact subset $U^L/G \subset V^L/G-D'$ with a (trivial) iterated circle fibration $U^L/G \times (S^1)^m \subset ((V^L/G) \times T)-D$ such that $U^L/G \times (S^1)^m$ and $((V^L/G) \times T)-D$ are homotopy equivalent.\\
Indeed, viewing $V^L/G \times T \subset V^L/G\times W$, the latter being a vector bundle, we have a zero section $V^L/G \subset V^L/G \times W$. Moreover, $\bar{D} \cap V^L/G$, where $\bar{D}$ is the closure of $D$ in $V^L/G \times W$, will be contained in some divisor $D'$. Shrinking $V^L/G-D'$ slightly, we can find a relatively compact open subset $U^L/G \subset V^L/G -D'$ homotopy equivalent to $V^L/G - D'$ and with the claimed circle fibration. 
\end{proof}

We say that the stable cohomology $H^{\ast}_{\mathrm{s}} (G, \ZZ/p )$ is detected by abelian subgroups if the map induced by the restriction to abelian subgroups
\[
H^*_{\mathrm{s}} (G, \: \ZZ /p\ZZ )  \longrightarrow \prod_{A} H^*_{\mathrm{s}} (A, \: \ZZ /p\ZZ )
\]
is injective (where $A$ ranges over all abelian subgroups of $G$). We will then also use the following principle which follows from Lemma \ref{lKuennethFactorAbelian}.

\begin{lemma}\xlabel{lKuennethstable}
Suppose $G_1$ and $G_2$ are finite groups such that at least one of $H^*_{\mathrm{s}} (G_1, \: \ZZ/p )$ or $H^*_{\mathrm{s}} (G_2, \: \ZZ/p )$ is detected by abelian subgroups. Then one has a K\"unneth formula in stable cohomology
\[
H^{\ast }_{\mathrm{s}} (G_1 \times G_2, \: \ZZ/p ) \simeq H^{\ast}_{\mathrm{s}} (G_1, \: \ZZ/p) \otimes H^{\ast}_{\mathrm{s}} (G_2, \: \ZZ/p)\, . 
\]
\end{lemma}

\begin{proof}
There is always the natural surjection
\[
H^{\ast}_{\mathrm{s}} (G_1, \: \ZZ/p) \otimes H^{\ast}_{\mathrm{s}} (G_2, \: \ZZ/p) \stackrel{p}{ \twoheadrightarrow } H^{\ast }_{\mathrm{s}} (G_1 \times G_2, \: \ZZ/p )\, .
\]
Without loss of generality, we can assume that abelian subgroups $A_i$ are a detecting family for the stable cohomology of $G_1$:
\[
H^*_{\mathrm{s}} (G_1, \: \ZZ /p ) \hookrightarrow  \prod_{i\in I} H^*_{\mathrm{s}} (A_i, \: \ZZ /p )\, .
\]
Now by Lemma \ref{lKuennethFactorAbelian}, there is an injection 
\[
H^{\ast}_{\mathrm{s}} (G_1, \: \ZZ/p) \otimes H^{\ast}_{\mathrm{s}} (G_2, \: \ZZ/p) \stackrel{i }{\hookrightarrow } \prod_{i } H^{\ast}_{\mathrm{s}} (A_i\times G_2, \: \ZZ/p )\, .
\]
But $i = (\prod \mathrm{res}_{A_i\times G_2}) \circ p$ where 
\[
\prod \mathrm{res}_{A_i\times G_2}\, :\, H^{\ast }_{\mathrm{s}} (G_1 \times G_2, \: \ZZ/p ) \to  \prod_{i } H^{\ast}_{\mathrm{s}} (A_i\times G_2, \: \ZZ/p )
\]
is the product of restriction maps. Hence $p$ is also injective.
\end{proof}

\begin{lemma}\xlabel{lNoUnramified}
Let $G$ be a finite group such that $H^i_{\mathrm{nr}} (G, \: \ZZ /p\ZZ ) = 0$ for all $i > 0$. Then every stable class $a \in H^*_{\mathrm{s}} (G, \: \ZZ/p\ZZ )$ is nontrivial on the centralizer $C(g)$ of some element $g\in G$, i.e. the restriction $\mathrm{res}\, :\, H^*_{\mathrm{s}} (G, \: \ZZ/p\ZZ ) \to H^*_{\mathrm{s}} (C(g) , \: \ZZ/p\ZZ )$ is nonzero. 
\end{lemma} 

\begin{proof}
With the notation established above we have the maps of groups
\[
I \subset \mathrm{Gal} (K_{\nu }) \subset \mathrm{Gal} (K) \twoheadrightarrow G\, ,
\]
and the image of the inertia subgroup $I$ in $G$ is cyclic, generated by $g$ say, and the image of the decomposition group $\mathrm{Gal} (K_{\nu })$ in $G$ belongs to the centralizer $C(g)$. As the residue map $\partial_{\nu }$ factors over $\mathrm{Gal} (K_{\nu })$, we obtain the assertion.
\end{proof}

Recall the exact sequence 
\[
1 \to I \to \Gamma_{K_{\nu}} \to \Gamma_{\kappa_{\nu }}\to 1
\]
where $I$ is the inertia subgroup of the decomposition group $\Gamma_{K_{\nu}}$ associated to the valuation $\nu$ of $K= \CC (V)^G$. The following Lemma allows one to increase the usefulness of Lemma \ref{lNoUnramified} in inductive arguments further.

\begin{lemma}\xlabel{lResidue}
Let $G$ be a finite group and let $a \in H^{\ast }_{\mathrm{s}} (G, \: \ZZ/p )$ be a stable class. For $\nu$ a divisorial valuation of $K$, the image of the topologically cyclic inertia subgroup $I$ in $G$ is cyclic, generated by $h$ say. There is a natural class $d_{\nu } (a) \in H^{n-1}_{\mathrm{s}} (Z_G(h), \: \ZZ/p )$ such that the residue $\partial_{\nu } (a) \in H^{n-1} (\Gamma_{\kappa_{\nu}}, \: \ZZ /p )$ is the pull-back of $d_{\nu } (a)$ to $\Gamma_{\kappa_{\nu}}$ via the maps $\Gamma_{\kappa_{\nu}} \subset \Gamma_{K_{\nu}} \simeq I \oplus \Gamma_{\kappa_{\nu}} \to Z(h)$. 
\end{lemma}

\begin{proof}
Consider the commutative diagram
\[
\xymatrix{ 1 \ar[r] & I \ar[d] \ar[r] \ar@{-->}[rdd]^f &\Gamma_{K_{\nu}} \ar[d] \ar[r] & \Gamma_{\kappa_{\nu}} \ar[d] \ar[r] \ar[r] \ar@{-->}[ldd]^g& 1 \\
1 \ar[r] & \langle h\rangle  \ar[r] & Z(h)  \ar[r] & Z(h)/\langle h \rangle  \ar[r] & 1\\
1 \ar[r] & \langle h\rangle \ar[u]^a \ar[r] & \langle h \rangle \times Z(h)  \ar[u]^b \ar[r] & Z(h) \ar[u]^c \ar[r] & 1}
\] 
where the arrow $a$ is the identity, $b$ is the couple $(i_{\langle h \rangle }, \: \mathrm{id}_{Z(h)})$ where $i_{\langle h \rangle }\, :\, \langle h \rangle \to Z(h)$ is the inclusion, and $c$ is the projection. Here $f$ and $g$ are defined as follows: the extension defining $\Gamma_{K_{\nu}}$ splits, $\Gamma_{K_{\nu}} \simeq \hat{\ZZ} \oplus \Gamma_{\kappa_{\nu}}$, so the map $\Gamma_{\kappa_{\nu}} \to Z(h)/\langle h \rangle $ lifts to a map $g_1\, :\, \Gamma_{\kappa_{\nu}} \to Z(h)$. The map $g$ is simply $(1, \: g_1)$. The map $f$ is $(f_1, \: 1)$ where $f_1\,:\, I \to \langle h \rangle $ is the natural map. As $\hat{\ZZ}$ has cohomological dimension one, $H^i (\Gamma_{K_{\nu}}, \: \ZZ/p ) \simeq H^{i-1} (\Gamma_{\kappa_{\nu}}, \: \ZZ/p) \oplus H^i (\Gamma_{\kappa_{\nu }}, \ZZ/p)$ and the residue map $\partial_{\nu}$ was defined by the restriction of $a$ to $\Gamma_{K_{\nu}}$ and projecting to $H^{i-1} (\Gamma_{\kappa_{\nu}}, \: \ZZ/p)$. By the commutativity of the diagram, we may thus define a class $d_{\nu } (a)$ with the requested properties as follows: we restrict $a\in H^n_{\mathrm{s}} (G, \: \ZZ/p)$ to $Z(h)$, and then take the pull-back $b^{\ast} (\mathrm{res}^G_{Z(h)}(a)) \in H^n_{\mathrm{s}} (\langle h \rangle \times Z (h), \: \ZZ/p )$ and project this unto the component $H^{n-1}_{\mathrm{s}} (Z(h), \: \ZZ/p)$ in the K\"unneth decomposition, using Lemma \ref{lKuennethFactorAbelian}. This defines  $d_{\nu } (a) \in H^{n-1}_{\mathrm{s}} (Z(h), \: \ZZ/p )$. 
\end{proof}

Define a subgroup $H$ of $G$ recursively to be an \emph{iterated centralizer} if it is the centralizer of an element in $G$ or a centralizer of an element inside another iterated centralizer. 

\begin{corollary}\xlabel{cInduction}
Assume that $G$ is such that each iterated centralizer has trivial unramified cohomology. Then any element $a\in H^n_{\mathrm{s}} (G, \: \ZZ/p)$ is nontrivial on some abelian $p$-subgroup. 
\end{corollary}

\begin{proof}
We use induction on the cohomological degree, hence assume that every element in $H^{i}_{\mathrm{s}} (H, \: \ZZ/p)$, for all $i<n$, and for all iterated centralizers $H$ in $G$, is nontrivial on some abelian subgroup. By assumption, we get from Lemma \ref{lResidue} that $d_{\nu} (a) \in H^{n-1}_{\mathrm{s}} (Z(h), \: \ZZ/p )$ is nontrivial for some $h$. Hence $d_{\nu } (a)$ is nontrivial on some abelian $p$-subgroup $A$ of $Z(h)$. By the construction of $d_{\nu } (a)$ in Lemma \ref{lResidue} we have that $a$ will then be nontrivial when restricted to $H^n_{\mathrm{s}} (\langle h, \: A\rangle , \: \ZZ/p )$ where $\langle h, \: A\rangle $ is the abelian subgroup of $G$ generated by $h$ and $A$. 
\end{proof}

\

The Steenrod power operations $Sq^i$, $P^j$ (see \cite{Steen}, \cite{A-M} II.2) are natural transformations
\begin{gather*}
Sq^i \, : \, H^j (X, \: \ZZ/2 ) \to H^{j+i} (X, \: \ZZ/2 ), \\ P^i \, :\, H^j (X,\: \ZZ/p ) \to H^{j+2i (p-1)} (X, \: \ZZ/p ) , \; p \; \mathrm{an} \; \mathrm{odd}\; \mathrm{prime} \, ,
\end{gather*}
on the category of CW-complexes with continuous maps $f\, : \, X\to Y$. By functoriality, applied to the map $\mathrm{B} G \to (V^L/G) -D$, where $D$ is some divisor, $Sq^i$, $P^j$ induce operations on $H^{\ast }_{\mathrm{s}} (G, \: \ZZ/p)$. 

\

For later use, we recall here the structure theorem for the cohomology of wreath products due to Steenrod \cite{Steen}, Section VII, see also \cite{Mann78}, Theorem 3.1 and \cite{A-M}, IV. 4, Theorem 4.1. We suppress the $\ZZ/p\ZZ$-coefficients in cohomology groups now, i.e. write $H^* (X)$ for $H^* (X, \ZZ /p \ZZ )$.

\begin{theorem}\xlabel{tSteenrod}
Let $H$ be a group, and let $H \wr \ZZ/p = (H)^p\rtimes \ZZ /p$ be the wreath product where $\ZZ/p\ZZ$ acts by cyclically permuting the copies of $H$. Let $\mathrm{id} \times \Delta^p\, :\, \ZZ/p\ZZ \times H \to \ZZ/p\ltimes (H)^p$ be the inclusion $(\mathrm{id} \times \Delta^p ) (z, \: a) = (z ; \: (a, \dots , a))$ ($p$-times $a$) and denote by $t\, :\, H^* (H^p ) \to H^* (H\wr \ZZ /p  )$ the transfer. Then the sequence 
\begin{gather*}
H^* (H^p ) \stackrel{t}{\rightarrow } H^* ( H \wr \ZZ /p ) \stackrel{( \mathrm{id} \times \Delta^p)^*}{\longrightarrow}  H^* ( \ZZ/p\ZZ \times H  )
\end{gather*}
is exact.

\

Moreover, for any $u\in H^{j} (H)$ there is a class $P (u) \in H^{jp} ( H \wr \ZZ/ p )$ (constructed by Steenrod) such that
\begin{itemize}
\item[(i)]
If $j \, :\, H^p \to \ZZ /p\ZZ \ltimes (H)^p$ is the natural inclusion, then $j^* (P(u)) = u^{\otimes p}$.
\item[(ii)]
In the K\"unneth decomposition of  $(\mathrm{id} \times \Delta^p )^* (P(u))$ in $H^* ( \ZZ/p\ZZ \times H  )$ we have
\[
(\mathrm{id} \times \Delta^p )^* (P(u)) = \Sigma w_k \otimes D_k (u)
\]
where $w_k$ is a generator of $H^k (\ZZ /p\ZZ )$ and $D_k \, :\, H^q (H) \to H^{pq -k} (H)$ are homomorphisms which satisfy
\item[(iii)]
\begin{gather*}
\beta D_{2k} (u) = D_{2k-1} (u),\; \beta D_{2k-1} (u) = 0, \; \beta D_0 (u) =0  
\end{gather*}
where $\beta$ is the Bockstein homomorphism, i.e. connecting homomorphism in the long exact sequence coming from the short exact sequence
\[
0 \to \ZZ / p \ZZ \to \ZZ / p^2 \ZZ \to \ZZ /p \ZZ \to 0 \, .
\]
The maps $D_k$ are used originally by Steenrod to define the Steenrod powers $P^i$, hence the Hopf algebras (Steenrod algebras) $\mathcal{A}(p)$. More precisely, 
\[
P^i (u) = (-1)^{i+ mj(j+1)/2} (m!)^{j} D_{(p-1)(j-2i)} (u) \in H^{j+ 2(p-1)i} (H, \: \ZZ /p \ZZ ) \, .
\]
\end{itemize}
In this setting, any $c \in H^* ( H\wr \ZZ /p )$ can be written as
\[
c = t( c_1) + c_2 \cdot P (c_3)
\]
with $c_1 \in H^* (H^p)$, $c_2 \in H^* (\ZZ /p \ZZ )$ and $c_3 \in H^* (H)$. 
\end{theorem}

Here we view $H^* ( H\wr \ZZ /p )$ as a module over $H^* (\ZZ /p \ZZ )$ via the cohomology pull-back induced by the surjection $H \wr \ZZ /p \to \ZZ /p \ZZ$. The Steenrod operations $P^i$ have the following formal properties:
\begin{itemize}
\item[(1)]
$P^0 = \mathrm{id}$.
\item[(2)]
If $\dim (x) = 2n$, then $P^n (x) = x^p$.
\item[(3)]
If $2i > \dim (x)$, then $P^i (x) =0$.
\item[(4)]
(Cartan formula) $P^i (x\cup y) = \sum_{j=0}^i P^j (x) \cup P^{i-j} (y)$.
\end{itemize}

\

A consequence of the Bloch-Kato conjecture (now a theorem by the work of Voevodsky, Rost and many others) is 

\begin{lemma}\xlabel{lCohomologyOperations}
The Steenrod cohomology operations $Sq^i$, $P^j$ are zero in stable cohomology $H^{\ast }_{\mathrm{s}} (G, \: \ZZ/p\ZZ )$.
\end{lemma}

\begin{proof}
Any stable class $a \in H^{\ast }_{\mathrm{s}} (G, \: \ZZ/p\ZZ ) \subset H^{\ast } (\mathrm{Gal} (K), \: \ZZ/p )$ arises as the pull-back from some torus $T \simeq (\CC^{\ast })^m$, more precisely, $a$ has a representative $a'$ in the cohomology of $H^{\ast } ((V^L/G) -D, \: \ZZ/p )$ and there is a regular map $f\, :\, (V^L/G) - D \to  T$ together with a class $\hat{a} \in H^{\ast } (T, \: \ZZ/p )$ with $f^{\ast } (\hat{a}) = a'$. This follows from the fact, which is a consequence of the Bloch-Kato conjecture, that $H^{\ast } (\mathrm{Gal}^{\mathrm{ab}} (K), \: \ZZ/p ) \to H^{\ast } (\mathrm{Gal} (K), \: \ZZ/p )$ is surjective where $\mathrm{Gal}^{\mathrm{ab}} (K)$ is the abelianized Galois group $\mathrm{Gal}(K)/[\mathrm{Gal}(K), \: \mathrm{Gal}(K)]$. 

Since the Steenrod power operations are trivial in the cohomology algebra of the torus $T$ (which is an exterior algebra on one-dimensional generators), the assertion of the Lemma will then follow from the functoriality of the cohomology operations.

\

It remains to explain in some more detail how from the surjection $H^{\ast } (\mathrm{Gal}^{\mathrm{ab}} (K), \: \ZZ/p ) \to H^{\ast } (\mathrm{Gal} (K), \: \ZZ/p )$ we get the map $f\, :\, (V^L/G) - D \to  T$. There is a finite abelian quotient $\mathrm{Gal} (K) \twoheadrightarrow A$ such that $a$ is induced from a class $a''$ in the cohomology of $A$. The group $A$ corresponds to an unramified abelian covering $\tilde{X}$ of some nonempty open affine subvariety $X \subset V^L/G$. The coordinate ring $\CC [\tilde{X}]$ contains an arbitrary finite dimensional representation of $A$ as the regular functions on an $A$-orbit in $\tilde{X}$ are precisely the regular representation $\CC [A]$ and this is also a subrepresentation (not only a quotient) because $A$ is reductive. In particular, embedding $A$ in a torus of diagonal matrices in some $\mathrm{GL}_m (\CC )$, one obtains a dominant regular map from an open subset $X'$ of $X= \tilde{X}/A$, hence some $(V^L/G) - D$ to a torus $T=(\CC^m)^L /A\simeq (\CC^{\ast })^m$ for which it holds by construction that the image of the class $a''$ in the cohomology of $T$ induces a representative $a' \in H^{\ast } ((V^L/G) -D, \: \ZZ/p )$ of $a$. 
\end{proof}

Thus the techniques used in the present article are mainly topological in flavour; for the connection to motivic cohomology and further developments the reader may consult \cite{Kahn-Su00}, \cite{Kahn11}, \cite{Ngu1}, \cite{Ngu2}, \cite{TY11}.

\section{Detection by elementary abelian $p$-subgroups}\xlabel{sDetection}

In this section we want to prove the following Theorem.

\begin{theorem}\xlabel{tDetectionElAb}
Let $p$ as always be an odd prime and $\mathfrak{A}_N$ the alternating group on $N$ letters. Then $H^*_{\mathrm{s}} (\mathfrak{A}_N, \, \ZZ/p )$ is detected by elementary abelian $p$-subgroups. 
\end{theorem}

We denote by $G_n=\ZZ/p\ZZ \wr \dots \wr \ZZ /p\ZZ$ ($n$ factors) the iterated wreath product of $n$ cyclic $p$-groups. This is the $p$-Sylow of $\mathfrak{A}_{p^n}$. If $N$ is arbitrary (not necessarily a power of $p$) expand it in base $p$:
\[
N = a_0 + a_1 p  + \dots + a_m p^m
\]
with $0\le a_j < p$, $a_m \neq 0$, and note that this gives rise to a natural inclusion
\[
i_{a_1, \dots , a_m} \, :\, \mathfrak{A}_{a_1, \dots , a_m}:= \prod_{1}^{a_1} \mathfrak{A}_{p} \times \prod_{1}^{a_2} \mathfrak{A}_{p^2} \times \dots \times \prod_{1}^{a_m} \mathfrak{A}_{p^m} \hookrightarrow \mathfrak{A}_{N}
\]
and that a $p$-Sylow subgroup in $\mathfrak{A}_N$ is given by the product of  $p$-Sylow subgroups in the factors in $\mathfrak{A}_{a_1, \dots , a_m}$. Hence it follows from Lemma \ref{lKuennethstable} that it is enough to prove Theorem \ref{tDetectionElAb} for $N=p^n$. 

\

First we will prove the weaker

\begin{theorem}\xlabel{tDetectionAb}
The stable cohomology $H^*_{\mathrm{s}} (\mathfrak{A}_N, \, \ZZ/p )$ is detected by abelian $p$-subgroups. 
\end{theorem}

The proof of Theorem \ref{tDetectionAb} will follow from the structures of centralizers in \emph{complete monomial groups}. 

\begin{definition}\xlabel{dCompleteMonomial}
Let $H$ be a group. The \emph{complete monomial group} of degree $m$ on $H$ is the group $\Sigma_m (H):=H\wr \mathfrak{S}_m= (H)^m \rtimes \mathfrak{S}_m$, where $\mathfrak{S}_m$ is the symmetric group on $m$ letters. \\
A \emph{monomial cycle} in $\Sigma_m (H)$ is an element of the form $((h_1, \dots , \: h_m ) ; \: \sigma )$ where $\sigma \in \mathfrak{S}_m$ is a cycle. The \emph{determinant class} of a monomial cycle is the conjugacy class in $H$ of the product $h_1\cdot \ldots \cdot h_m$. The \emph{length} of a monomial cycle is the length of the underlying cycle $\sigma$.
\end{definition}

We have to recall some results from \cite{Ore} on the structure of conjugacy classes and centralizers in groups $\Sigma_m (H)$:

\begin{itemize}
\item[(1)]
Two monomial cycles in $\Sigma_m (H)$ are conjugate if and only if they have the same length and determinant class (\cite{Ore}, Theorem 6).
\item[(2)]
Any element $x\in \Sigma_m (H)$ can be written uniquely as a product of commuting monomial cycles (the underlying cyles in $\mathfrak{S}_m$ have no common variables) (\cite{Ore}, Theorem 3). 
\item[(3)]
From (1) and (2) follows the description of conjugacy classes in $\Sigma_m (H)$: two elements $x'$, $x$ in $\Sigma_m (H)$ are conjugate if the monomial cycles in their decompositions in (2) can be matched in such a way that corresponding cycles have the same length and determinant class.
\item[(4)]
Let $c=((h, \ldots , \: 1, \: 1) ; \: \sigma)$ be a monomial cycle of length $m$ in $\Sigma_m (H)$. Then its centralizer in $\Sigma_m (H)$ is an extension 
\[
1\to Z_H (h) \to Z_{\Sigma_m (H)} (c) \to \ZZ/m \to 0
\]
(\cite{Ore}, p.20, 21). In other words, the centralizer $Z_{\Sigma_m (H)} (c)$ is generated by $c$ and the group $Z_H(h)$ embedded diagonally into $H^m\subset \Sigma_m (H)$.
\item[(5)]
(\cite{Ore}, Theorem 8): Let
\[
x = x_1 \cdot \ldots \cdot x_l , \; x_i = y_{1}^{(i)} \cdot \ldots \cdot y_{k_i}^{(i)}
\]
be the decomposition of an element $x$ in $\Sigma_m (H)$ into disjoint monomial cycles $y_{j}^{(i)}$ as in (2), where we group the cycles of equal length and determinant class together: for fixed $i$, all $y_j^{(i)}$ have determinant class $h_i$ and length $n_i$. Then the centralizer of $x$ in $\Sigma_m (H)$ has a description as
\[
Z_{\Sigma_m (H)} (x) = \prod_{i=1}^{l} \Sigma_{k_i} (Z_{\Sigma_{n_i}(H)} (y_1^{(i)}))
\]
where as in (4) the centralizer of $y_1^{(i)}$ (or any $y_j^{(i)}$) in the group $\Sigma_{n_i}(H)$ is an extension
\[
1\to Z_{H} (h_i) \to Z_{\Sigma_{n_i}(H)} (y_1^{(i)}) \to \ZZ/n_i \ZZ \to 0\, .
\]
\end{itemize}

From the last fact (5) we get immediately

\begin{lemma}\xlabel{lCentralizers1}
Let $\Sigma_m (A)$ be a complete monomial group with $A$ abelian. Then the centralizer of any element $x$ in $\Sigma_m (A)$ is a product of groups of the same type
\[
Z_{\Sigma_m (A)} (x) = \prod_{h=1}^M \Sigma_{k_h} (A_h)
\]
where the $A_h$ are abelian.
\end{lemma}

\begin{proof}
It suffices to remark that any central extension of a cyclic group by an abelian group is again abelian. 
\end{proof}

To prove Theorem \ref{tDetectionAb} it suffices, by the technique of Lemma \ref{lResidue} exposed above, to show the following.

\begin{lemma}\xlabel{lCentralizers2}
Let $p$ be an odd prime and let $x$ be an element of order a power of $p$ in a group $A\wr \mathfrak{A}_m=A^m \rtimes \mathfrak{A}_m$ where $A$ is an abelian $p$-group. Then there is a group
\[
Z' =  \prod_{h=1}^M A_h \wr \mathfrak{A}_{k_h}
\]
where all the $A_h$ are abelian $p$-groups, and with the property that $Z'$ is contained in the centralizer $Z_{A\wr \mathfrak{A}_m} (x)$ and contains a $p$-Sylow of $Z_{A\wr \mathfrak{A}_m} (x)$.
\end{lemma}

\begin{proof}
We consider the group $A\wr \mathfrak{A}_m$ as a subgroup of the complete monomial group $\Sigma_m (A)$. By Lemma \ref{lCentralizers1} itsuffices to determine the intersection of $Z_{\Sigma_m (A)} (x) = \prod_{h=1}^M \Sigma_{k_h} (A_h)$ and $A\wr \mathfrak{A}_m$. As $p$ is odd, it is clear that the intersection contains
\[
Z' = \prod_{h=1}^M A_h \wr \mathfrak{A}_{k_h}
\]
and that the complete centralizer $Z_{A\wr \mathfrak{A}_m} (x)$ is an extension 
\[
1 \to Z' \to Z_{A\wr \mathfrak{A}_m} (x) \to (\ZZ/2)^r \to 0
\] 
where $(\ZZ/2)^r$ is an elementary abelian $2$-group which can be identified with the kernel of the sign
\[
\prod_{h=1}^M \mathfrak{S}_{k_h} \subset \mathfrak{S}_{\sum k_h} \to \{\pm 1\}
\]
modulo the subgroup $\prod_{h=1}^M \mathfrak{A}_{k_h}$. The statement follows as $p$ is odd.
\end{proof}

Thus we obtain

\begin{proof}(of Theorem \ref{tDetectionAb})
We will prove more generally that $H^*_{\mathrm{s}} (G, \: \ZZ /p)$ is detected by abelian $p$-subgroups where $G$ is any group which is a product of groups $A\wr \mathfrak{A}_m$ with $A$ an abelian $p$-group.
We have that $H^*_{\mathrm{s}} (G, \: \ZZ /p)$ is detected by $H^*_{\mathrm{s}} ( \mathrm{Syl}_p (G), \: \ZZ /p )$, and the higher unramified cohomology of $\mathrm{Syl}_p (G )$ is trivial. This follows immediately from \cite{B-P}, Lemma 2.4, namely, if one forms a wreath product of groups, each of which has stably rational generically free linear quotients, then the wreath product inherits this property.\\
Hence every element $a$ in  $H^n_{\mathrm{s}} (G, \: \ZZ /p)$ will, in the notation of Lemma \ref{lResidue}, give a nontrivial $d_{\nu } (a) \in H^{n-1}_{\mathrm{s}} (Z_{G} (h), \: \ZZ /p )$ for some element $h\in G$ of $p$-power order. Thus it will be enough to show that $H^{n-1}_{\mathrm{s}} (Z_{G} (h), \: \ZZ /p )$ is detected by abelian $p$-subgroups. But $H^{n-1}_{\mathrm{s}} (Z_{G} (h), \: \ZZ /p )$ is detected by $H^{n-1}_{\mathrm{s}} (\mathrm{Syl}_p (Z_{G} (h)), \: \ZZ /p )$ and, by Lemma \ref{lCentralizers2}, $\mathrm{Syl}_p (Z_{G} (h))$ is contained in a group which in turn is contained in $Z_G (h)$ and is again a product of groups of type $A\wr \mathfrak{A}_m$. Hence we can conclude by induction on the cohomological degree $n$.
\end{proof}

Now we prove Theorem \ref{tDetectionElAb}. It will follow immediately from 

\begin{proposition}\xlabel{pReductionToSmaller}
Let $N=p^n$, and suppose that $A$ is an abelian $p$-subgroup of $\mathfrak{A}_N$. Thus one can write
\[
A = \prod_{i=1}^k (\ZZ / (p^{l_i}))^{r_i}, \: l_i, \: r_i \in \mathbb{N}\, .
\]
If $A$ is not reduced to a single cyclic group $\ZZ/(p^l)$, then $A$ is contained in a product of alternating groups $\prod_{j=1}^h \mathfrak{A}_{t_h} \subset \mathfrak{A}_N$ with $t_h < N$ for all $h$.
\end{proposition}

Once we have this Proposition, the proof of Theorem \ref{tDetectionElAb} is an induction: it suffices to prove it for $N=p^n$, and we may assume that detection by elementary abelian subgroups holds for the stable cohomology of all $\mathfrak{A}_j$ with $j<N$. Now clearly, $H^1_{\mathrm{s}} (\mathfrak{A}_N, \: \ZZ /p)$ is detected by elementary abelian $p$-subgroups, for $H^1 (\mathfrak{A}_N, \: \ZZ /p )$ can be identified with characters $\chi\, :\, \mathfrak{A}_N \to \ZZ/ p$ whence $ H^1_{\mathrm{s}} (\mathfrak{A}_N, \: \ZZ /p) =0$ unless $p=3$ and $N=3$ so that $\mathfrak{A}_N = \ZZ/3$. So we have to show that any stable class $a\in H^i_{\mathrm{s}} ( \mathfrak{A}_N, \: \ZZ/ p )$ for $i\ge 2$ is nontrivial on an elementary abelian $p$-subgroup. By Theorem \ref{tDetectionAb} $a$ is nontrivial on an abelian $p$-subgroup $A$ with $\mathrm{rk} (A/pA) \ge 2$ (as the stable cohomology of $A$ is an exterior algebra on $\mathrm{rk} (A/pA)$ generators). Such an $A$ is contained in a product of smaller alternating groups by Proposition \ref{pReductionToSmaller}. Thus the proof is complete by induction.

\begin{proof} (of Proposition \ref{pReductionToSmaller})

Denote by $X_N = \{ 1, \dots ,\: N \}$ the set of letters on which the ambient $\mathfrak{S}_N \supset A$ acts. Let
\[
X_N = \coprod_{\alpha} X_{N, \: \alpha}
\]
be the decomposition of $X_N$ into $A$-orbits. Let $X_{N, \alpha_0 } =: X$ be a fixed orbit. This orbit is isomorphic to a quotient $\bar{A}$ of $A$, hence a group of the same form 
\[
\bar{A} = \prod_{i=1}^{\bar{k}} (\ZZ / (p^{\bar{l}_i}))^{\bar{r}_i}
\]
and the action of $A$ on this orbit is via \emph{the regular representation} of $\bar{A}$ on itself. In other words, $A$ embeds into a subgroup $\prod_{\alpha } \bar{A}_{\alpha }$ of $\mathfrak{A}_N$ where each $\bar{A}_{\alpha }$ is embedded into a subgroup $\mathfrak{A}_{\mathrm{ord} (\bar{A}_{\alpha })}$ via the regular representation.\\
In summary, it suffices to prove the statement of Proposition \ref{pReductionToSmaller} for the case that the group $A$ in its statement is embedded into the ambient $\mathfrak{A}_N$ via the regular representation. We can write $A = A' \times \ZZ /(p^k)$ with $\mathrm{rk} (A' /pA')  < \mathrm{rk} (A/pA)$. Moreover, by definition of the regular representation, the composition of arrows 
\[
A = A' \times \ZZ / (p^k) \hookrightarrow \mathfrak{A}_{| A' | } \times \ZZ / (p^k) \hookrightarrow  \mathfrak{A}_{| A' | } \wr \ZZ / (p^k) \hookrightarrow \mathfrak{A}_{|A'| \cdot p^k}\simeq \mathfrak{A}_{|A|}
\]
gives the regular representation of $A$ where the first arrow $\hookrightarrow$ from the left is induced by the regular representation of $A'$, the second such arrow embeds $ \mathfrak{A}_{| A' | } \times \ZZ / (p^k)$ into the wreath product $\mathfrak{A}_{| A' | } \wr \ZZ / (p^k)$ by sending $(a; \sigma )$ to $(a, a, \dots , a; \sigma)$ as usual, and the last arrow embeds the wreath product $\mathfrak{A}_{| A' | } \wr \ZZ / (p^k)$  into $\mathfrak{A}_{|A'| \cdot p^k}$ by partitioning the set of $|A'| \cdot p^k$ objects which $\mathfrak{A}_{|A'| \cdot p^k}$ permutes into $p^k$ disjoint groups of $|A'|$ objects, and letting $\ZZ/(p^k)$ act by cyclically rotating these groups, and letting $(\mathfrak{A}_{|A'|})^{p^k}$ act via permutations within these groups. It follows that
\[
A \subset \mathfrak{A}_{| A' |} \times \mathfrak{A}_{p^k}
\]
where now $\mathfrak{A}_{p^k}$ is embedded into $\mathfrak{A}_{|A|}$ as arbitrary alternating (not only cyclic) permutations of the $p^k$ groups of items. Note that elements of the two subgroups $\mathfrak{A}_{p^k}$ and $\mathfrak{A}_{| A' |}$ of the group $ \mathfrak{A}_{|A|}$ commute, and the two subgroups intersect trivially, so that we do have a direct product. Moreover, if $A$ is not reduced to a single cyclic group, we have that $A'$ is not the trivial group, and $p^k < | A|$.
\end{proof}

\section{Stable cohomology of alternating groups}\xlabel{sStableAlternating}

Let $\mathfrak{A}_n$ be, as in the previous section, the alternating group on $n$ letters, and let $p$ be an odd prime (the case $p=2$ has been treated in \cite{B-P}).  We assume first $n=p^m$ for simplicity.

\

We have to know the way elementary abelian $p$-subgroups sit inside $\mathfrak{A}_n$ for the following. We summarize everything in the following Lemma which is proven by arguments analogous to those already used in the proof of Proposition \ref{pReductionToSmaller}. 

\begin{lemma}\xlabel{lConjugacySymmetric}
Suppose $n=p^m$ and denote by $I_m:= \{ \underline{i}=(i_1, \dots , i_m)\in  \mathbb{N}^m \}$ the set of all nonnegative integer sequences $\underline{i}$ with
\[
p^m = i_1 p + i_2 p^2 + \dots + i_m p^m = \sum_{j=1}^m i_j p^j \, .
\] 
Then there is a natural bijection between $I_m$ and the set of conjugacy classes of maximal elementary abelian $p$-subgroups in $\mathfrak{S}_{p^m}$. The subgroup $T (i_1, \dots , i_m)$ corresponding to $\underline{i}$ can be described as follows: partition the set of integers $X = \{ 1, \dots , n\}$ into segments of $p$ power lengths according to $\underline{i}$:
\[
X = \bigcup_{j=1}^m \bigcup_{s=1}^{i_j} X^j_s
\]
where $X^j_s$ is a set with $p^j$ elements, 
\[
X^j_s = \{ i_1p + \dots + i_{j-1}p^{j-1} + (s-1)p^j, \dots , \: i_1p + \dots + i_{j-1}p^{j-1} + s p^j  \}
\]
for definiteness. The subset $X^j_s$ corresponds to a subgroup $\mathfrak{S}_{p^j} = \left( \mathfrak{S}_{p^j} \right)^{X^j_s}\subset \mathfrak{S}_{p^m}$  fixing all elements in $X$ outside $X^j_s$. Inside $\left( \mathfrak{S}_{p^j} \right)^{X^j_s}$ there is a copy of $(\ZZ/p \ZZ )^j$, which we denote by $\left(  (\ZZ/p\ZZ )^j \right)^{X^j_s}$, embedded via the regular representation, i.e. we identify the elements in $X^j_s$ with the elements of $\left( (\ZZ/p\ZZ)^j \right)^{X^j_s}$ and the permutation action is then given by left multiplication. 

\

We denote $T(0, \dots, p^{m-k}, \dots , 0)$ (a single nonzero entry $p^{m-k}$ in the $k$-th place) by $T_{k,m}$.

\

Hence every maximal elementary abelian $p$-subgroup in $\mathfrak{A}_n$ is conjugate -in $\mathfrak{S}_n$ or $\mathfrak{A}_n$, it is the same thing- to one contained in $\mathfrak{A}_{p^{n-1}} \times \dots \times \mathfrak{A}_{p^{n-1}}$ ($p$ factors) or conjugate to $T_{m,m}$.
\end{lemma}

The proof is immediate if one notices that under the action of some elementary abelian $p$-subgroup $A$ the set $X$ breaks up into $A$ orbits of cardinality a $p$ power, and the action of $A$ restricted to an orbit embeds $A$ into the permutation group of the elements of the orbit in such a way that the image is conjugate to the image of the regular representation. The result is in \cite{A-M} VI. 1, Thm. 1.3, but also \cite{Mui}, Chapter II, \S 2, where it is ascribed to Dixon. For the statement that the conjugacy classes of maximal elementary abelian $p$-subgroups in $\mathfrak{A}_n$ are the same as in $\mathfrak{S}_n$ one can appeal to the following Lemma which we will also use in other instances below (it is e.g. in \cite{Mann85}, p. 269).

\begin{lemma}\xlabel{lNormalizerTmm}
For $n=p^m$ the Weyl groups $W_{\mathfrak{S}_n} (T_{m,m}) = N_{\mathfrak{S}_n}(T_{m, m})/T_{m, m}$ resp. $W_{\mathfrak{A}_n} (T_{m,m})$ of $T_{m,m} \simeq (\ZZ/p\ZZ)^m$ inside $\mathfrak{S}_n$ resp. $\mathfrak{A}_n$ are 
\begin{gather*}
W_{\mathfrak{S}_n} (T_{m,m}) = \mathrm{GL}_m (\mathbb{F}_p), \; W_{\mathfrak{A}_n} (T_{m,m}) = \mathrm{GL}^+_m (\mathbb{F}_p)
\end{gather*}
where $\mathrm{GL}^+_m (\mathbb{F}_p)$ is the kernel of the map $\mathrm{GL}_m (\mathbb{F}_p ) \to \ZZ/2\ZZ$ given by the determinant raised to the power $(p-1)/2$.
\end{lemma}

In fact it is true that the Weyl group of any group $H$ in the embedding $H \hookrightarrow \mathfrak{S}_{|H|}$ given by the regular representation is the group of outer automorphisms of $H$, which become all inner in $\mathfrak{S}_{|H|}$. Both statements of the Lemma follow from this remark as $\mathrm{Aut} ((\ZZ/p\ZZ)^m) = \mathrm{GL}(m, \: \mathbb{F}_p)$. Likewise, Lemma \ref{lNormalizerTmm} implies that in the normalizer of any maximal elementary abelian $p$-subgroup in $\mathfrak{S}_n$ there are elements which do not lie in $\mathfrak{A}_n$. Hence conjugacy classes of these in the two groups coincide.

\

We will also use in an essential way the C\'{a}rdenas-Kuhn Theorem to calculate the stable cohomology of $\mathfrak{A}_n$, so we recall the precise statement (see \cite{A-M} III.5 for the proof).

\begin{theorem}\xlabel{tCardenasKuhn}
Let $E \subsetneqq S \subsetneqq G$ be a \emph{closed} system of finite groups, where the closedness means that every subgroup of $S$ which is conjugate to $E$ in $G$ is already conjugate to $E$ in $S$. Let $W_G (E)= N_G(E)/E$ resp. $W_S(E) = N_S (E)/E$ be the Weyl groups of $E$ in $G$ resp. $S$, and suppose that $E$ is $p$-elementary and that $W_S (E)$ contains a $p$-Sylow of $W_G(E)$. \\
Then the image of the restriction map  
\[
\mathrm{res}^G_E \, :\, H^* (G, \: \ZZ /p\ZZ ) \to H^* (E, \: \ZZ /p\ZZ )
\]
is equal to
\[
\mathrm{im} \left( \mathrm{res}^S_E \, :\,  H^* (S, \: \ZZ /p\ZZ ) \to H^* (E, \: \ZZ /p\ZZ ) \right) \cap H^* (E, \: \ZZ/p\ZZ )^{W_G (E)} \, .
\]
\end{theorem}

We will mostly use this in the form of the following 

\begin{corollary}\xlabel{cCardenasKuhn}
Let $S$ be a $p$-Sylow of  a finite group $G$, and let $E$ be an elementary abelian $p$-subgroup of $S$. Suppose that any subgroup of $S$ conjugate to $E$ in $G$ is conjugate to $E$ in $S$. Then we have
\begin{gather*}
\mathrm{im}\left( \mathrm{res}^{G}_E \, : \, H^* (G, \: \ZZ/p\ZZ) \to H^* (E, \: \ZZ/p\ZZ ) \right)\\ = \mathrm{im}\left( \mathrm{res}^{S}_E \, : \, H^* (S, \: \ZZ/p\ZZ) \to H^* (E, \: \ZZ/p\ZZ ) \right) \cap H^* (E, \: \ZZ/p\ZZ )^{W_G (E)}\, .
\end{gather*}
\end{corollary}

\begin{proof}
It suffices to remark that $[G : S] \equiv [N_G(E) : N_S(E)]\not\equiv 0$ (mod $p$). 
\end{proof}

\

Now if $n$ is arbitrary (not necessarily a power of $p$), to understand $H^*_{\mathrm{s}} (\mathfrak{A}_{n}, \: \ZZ /p\ZZ )$, expand $n$ in base $p$:
\[
n = a_0 + a_1 p  + \dots + a_m p^m
\]
with $0\le a_j < p$, $a_m \neq 0$, and note that this gives rise to a natural inclusion
\[
i_{a_1, \dots , a_m} \, :\, \mathfrak{A}_{a_1, \dots , a_m}:= \prod_{1}^{a_1} \mathfrak{A}_{p} \times \prod_{1}^{a_2} \mathfrak{A}_{p^2} \times \dots \times \prod_{1}^{a_m} \mathfrak{A}_{p^m} \hookrightarrow \mathfrak{A}_{n}
\]
and that a $p$-Sylow subgroup in $\mathfrak{A}_n$ is given by the product of  $p$-Sylow subgroups in the factors in $\mathfrak{A}_{a_1, \dots , a_m}$. In the notation of Lemma \ref{lConjugacySymmetric}, the group $\mathfrak{A}_{a_1, \dots , a_m}$ contains an elementary abelian $p$-subgroup
\begin{gather*}
E:= \prod_1^{a_1} T_{1,1} \times \prod_{1}^{a_2} T_{1,2} \times \dots \times \prod_{1}^{a_m} T_{1,m} \simeq ( \ZZ/p \ZZ )^{\frac{n-a_0}{p}} \, .
\end{gather*}

\begin{proposition}\xlabel{pDetectionAlternating}
The group $E$ detects the stable cohomology of $\mathfrak{A}_n$, i.e. 
\[
H^*_{\mathrm{s}} (\mathfrak{A}_n, \: \ZZ /p\ZZ ) \to H^*_{\mathrm{s}} (E , \: \ZZ /p\ZZ )
\]
is injective. 
\end{proposition}

\begin{proof}
It will be sufficient to prove this for $n=p^m$ as a K\"unneth theorem holds in stable cohomology for groups whose stable cohomology is detected by abelian subgroups, cf. Lemma \ref{lKuennethstable}. Now $\mathfrak{A}_{p^m}$ contains the wreath product
\[
\mathfrak{A}_{p^{m-1}}\wr \ZZ/ p \ZZ 
\]
which detects the stable cohomology of $\mathfrak{A}_{p^m}$ as it contains a $p$-Sylow. Using induction, it will be sufficient to prove that
\[
H^*_{\mathrm{s}} (\mathfrak{A}_{p^{m}}, \: \ZZ /p\ZZ ) \to H^*_{\mathrm{s}} ( \mathfrak{A}_{p^{m-1}}\times \dots \times \mathfrak{A}_{p^{m-1}})
\]
is injective for $m>1$.  By Lemma \ref{lConjugacySymmetric} and because $H^*_s (\mathfrak{A}_{n}, \: \ZZ/p\ZZ )$ is detected by elementary abelian $p$-subgroups, it will be sufficient to show that all positive-dimensional classes in $H^* (T_{m, m}, \: \ZZ /p \ZZ )$ coming as restrictions from $H^* (\mathfrak{A}_{p^m}, \: \ZZ / p\ZZ )$ are unstable. This follows from the calculation in \cite{Mann85}, Theorem 1.9, and the fact that the Bocksteins are zero in stable cohomology.
\end{proof}

\begin{theorem}\xlabel{tStableCohomologyAlternating}
Let $p$ be an odd prime as before. Then $H^*_{\mathrm{s}} (\mathfrak{A}_n, \: \ZZ/p\ZZ ) = 0$ in positive degrees unless $p=3$. For $p=3$ one has for $k\in\mathbb{N}$
\[
H^*_{\mathrm{s}} (\mathfrak{A}_{3k}, \: \ZZ/3\ZZ ) \simeq H^*_{\mathrm{s}} (\mathfrak{A}_{3k+1}, \: \ZZ/3\ZZ ) , \quad H^d_{\mathrm{s}} (\mathfrak{A}_{3k+2}, \: \ZZ/3\ZZ ) = 0 , \: d>0 \, ,
\] 
and 
\begin{gather*}
H^d_{\mathrm{s}} (\mathfrak{A}_{3k}, \: \ZZ/3\ZZ ) \neq 0 \, \mathrm{for}\; d>0 \iff d= k , \; \mathrm{and} \\
H^k_{\mathrm{s}} (\mathfrak{A}_{3k}, \: \ZZ/3\ZZ ) \simeq \langle \mathrm{det}_k \rangle \; \mathrm{where} \; \mathrm{res}^{\mathfrak{A}_{3k}}_{E} (\mathrm{det}_k) = e_1\wedge \dots \wedge e_k 
\end{gather*}
where $H^*_{\mathrm{s}} ( E, \:Ê\ZZ/3\ZZ ) = H^*_{\mathrm{s}} ((\ZZ/3\ZZ )^k, \: \ZZ/3\ZZ )$ is an exterior algebra on one-dimensional generators $e_1, \dots, e_k$. 
\end{theorem}

Basically, we would like to use the Card\'{e}nas-Kuhn Theorem with the elementary abelian subgroup $E$, and $S= \mathrm{Syl}_p (\mathfrak{A}_n)$, $G= \mathfrak{A}_n$, but it will be more transparent to break it up into several steps.

\begin{lemma}\xlabel{lVanishingAlternating}
For $p\neq 3$ an odd prime we have in positive degrees  $H^*_{\mathrm{s}} (\mathfrak{A}_n, \: \ZZ/p\ZZ ) = 0$.
\end{lemma}

\begin{proof}
The Weyl group $W_{\mathfrak{A}_n}(E)$ contains two obvious subgroups: (1) the group $\mathfrak{A}_{N}$ permuting the $N:=(n-a_0)/p$ copies of $\ZZ/p\ZZ$ in $E$, (2) a product $\prod_{1}^{N} (\ZZ/ p\ZZ )^{* , +}$ where $(\ZZ/p\ZZ)^{*, +}$ is the subgroup of the group of units in $\ZZ/p\ZZ$ given as the kernel of $a\mapsto a^{(p-1)/2}$. The stable cohomology of $E$ is an exterior algebra over $\ZZ /p\ZZ$ on $N$ generators $e_1, \dots , e_N$. The $\mathfrak{A}_N$-invariants are concentrated in degrees $0, \: 1, \: (N-1), \: N$, one-dimensional in each case and generated by
\[
1, \quad e_1 + \dots + e_N, \quad f_1\wedge \dots \wedge f_{N-1} , \quad e_1 \wedge \dots \wedge e_N \, ,
\]
where $f_1, \dots ,\: f_{N-1}$ is a basis of the $\mathfrak{A}_{N}$-invariant complement to $e_1+\dots + e_N$ in $H^1 (E, \: \ZZ/p\ZZ )$. All of these are not invariant under the scalings in $\prod_{1}^{N} (\ZZ/ p\ZZ )^{* , +}$ unless $p=3$ when $(\ZZ/ p\ZZ )^{* , +}$ is reduced to $\{ 1\}$. 
\end{proof}

\begin{lemma}\xlabel{lPeriodicityAlternating}
One has 
\begin{itemize}
\item[(1)]
$H^d_{\mathrm{s}} (\mathfrak{A}_{3k+2}, \: \ZZ/3\ZZ ) = 0 , \: d>0 \, .$
\item[(2)]
There is a natural embedding
\[
H^*_{\mathrm{s}} (\mathfrak{A}_{3k+1}, \: \ZZ/3\ZZ )\hookrightarrow H^*_{\mathrm{s}} (\mathfrak{A}_{3k}, \: \ZZ/3\ZZ )\, .
\]
\end{itemize}
\end{lemma}

\begin{proof}
This is already contained in \cite{B-P}, Lemmas 4.1 and 4.2. For completeness, let us repeat the argument: the restriction $\mathrm{res}^{\mathfrak{A}_{3k+2}}_{\mathrm{Syl}_3 (\mathfrak{A}_{3k+2} )}$ factors through the restriction map induced from the embedding $\mathfrak{S}_{3k} \hookrightarrow \mathfrak{A}_{3k+2}$; but $H^*_{\mathrm{s}} (\mathfrak{S}_{3k}, \: \ZZ/3\ZZ ) =0$ in positive degrees as the stable cohomology of $\mathfrak{S}_{3k}$ is detected by its elementary abelian $2$-subgroup generated by a maximal set of commuting transpositions. This proves (1), and (2) follows from the fact that the $3$-Sylows in $\mathfrak{A}_{3k}$ and $\mathfrak{A}_{3k+1}$ are the same.
\end{proof}

\begin{lemma}\xlabel{lWeylGroupAlternating}
Let $n= 3k$ or $n=3k+1$. Then the Weyl group $N_{\mathfrak{A}_n} (E)$ of $E\simeq (\ZZ/3\ZZ )^k$ in $\mathfrak{A}_n$ sits in an extension
\[
1 \to ( \ZZ/2\ZZ)^{k-1}    \to W_{\mathfrak{A}_n} (E) \to \mathfrak{S}_{k}  \to 1
\]
where  $\mathfrak{S}_k$ acts by permuting the copies of $\ZZ/3\ZZ$ in $E\simeq ( \ZZ/3\ZZ )^k$, and the group $(\ZZ/2\ZZ )^{k-1}$ acts by sending an \emph{even} number of the generators $g_i$ in the $i$th copy of $\ZZ/3\ZZ$ to their inverses $g_i^{-1}$. In stable cohomology $H^*_{\mathrm{s}} (E, \: \ZZ/3\ZZ ) = E(e_1, \dots ,\: e_k)$ (exterior algebra), the action of the group $W_{\mathfrak{A}_n} (E)$ is generated by the (signed) transpositions sending $e_i, \: e_j$ to $e_j, \: -e_i$, and transformations corresponding to elements in $(\ZZ/2\ZZ )^{k-1}$ acting via sign changes $e_i\mapsto -e_i$ on an \emph{even} number of the $e_i$.
\end{lemma}

\begin{proof}
Any element of the normalizer $N_{\mathfrak{A}_n}(E)$ induces a well-defined permutation of the copies of $\ZZ/3\ZZ$ in $E$. This gives a map to $\mathfrak{S}_k$ which is \emph{onto}: note that conjugating the $3$-cycle $(123)$ by $\tau = (23)$ exchanges the two nontrivial elements $g, \: g^{-1}$ in $\ZZ/3\ZZ = \langle (123) \rangle$, so we can also transpose two copies of $\ZZ/3\ZZ$ in $E$ by conjugating by an element in $\mathfrak{A}_n$. Now suppose $n \in N_{\mathfrak{A}_n}(E)$ induces the trivial element in $\mathfrak{S}_k$, so fixes all the copies of $\ZZ/3\ZZ$ in $E$ (though not necessarily elementwise). Then the only possible nontrivial automorphism of each copy of $\ZZ/3\ZZ$ is exchanging $g$ and $g^{-1}$ as before. To conclude the proof, it suffices to note that  if $n$ induces the identity in $\mathrm{Aut} ((\ZZ/3\ZZ)^k)$, then $n\in (\ZZ/3\ZZ )^k$.
\end{proof}

We can now turn to the 

\begin{proof}(of Theorem \ref{tStableCohomologyAlternating})
The remaining assertion not covered by Lemma \ref{lVanishingAlternating} and Lemma \ref{lPeriodicityAlternating} are that 
\begin{gather}\label{formula1}
H^d_{\mathrm{s}} (\mathfrak{A}_{3k}, \: \ZZ/3\ZZ ) \neq 0 \, \mathrm{for}\; d>0 \iff d= k , \; \mathrm{and} \\
H^k_{\mathrm{s}} (\mathfrak{A}_{3k}, \: \ZZ/3\ZZ ) \simeq \langle \mathrm{det}_k \rangle \; \mathrm{where} \; \mathrm{res}^{\mathfrak{A}_{3k}}_{E} (\mathrm{det}_k) = e_1\wedge \dots \wedge e_k \, . \nonumber
\end{gather}
and that
\begin{gather}\label{formula2}
H^*_{\mathrm{s}} (\mathfrak{A}_{3n+1} , \: \ZZ /3\ZZ ) \to H^*_{\mathrm{s}} (\mathfrak{A}_{3n}, \: \ZZ/3\ZZ )
\end{gather}
is \emph{surjective} (it is injective by (2) of Lemma \ref{lPeriodicityAlternating}). 

\

We prove first the assertions in the displayed formula \ref{formula1} above, and \ref{formula2} will follow easily (we just have to check that the determinant class comes from $H^*_{\mathrm{s}} (\mathfrak{A}_{3n+1} , \: \ZZ /3\ZZ )$). 
We apply the Card\'{e}nas-Kuhn Theorem \ref{tCardenasKuhn} with $S=\mathrm{Syl}_3 (\mathfrak{A}_{3k})$ containing $E$ and $G= \mathfrak{A}_{3k}$. Then
\begin{itemize}
\item
The fact that $E\simeq (\ZZ/3\ZZ )^k \subset \mathrm{Syl}_3 (\mathfrak{A}_{3k}) \subset \mathfrak{A}_{3k}$ is a closed system has been checked in \cite{Mui}, Prop. 2.2: in fact, he checks that if $A$ is any maximal elementary abelian $p$-subgroup of a symmetric group $\mathfrak{S}_n$, then any subgroup of a $p$-Sylow $\mathrm{Syl}_p (\mathfrak{S}_n)$ containing $A$ which is conjugate to $A$ in $\mathfrak{S}_n$ is conjugate to $A$ in $\mathrm{Syl}_p (\mathfrak{S}_n)$. This implies clearly the statement for the alternating groups we need.
\item
By the Card\'{e}nas-Kuhn Theorem or rather its Corollary \ref{cCardenasKuhn}, we get that the image of the cohomology of $\mathfrak{A}_{3k}$ in the cohomology of $E$ is
\begin{gather*}
\mathrm{im}\left( \mathrm{res}^{\mathrm{Syl}_3 (\mathfrak{A}_{3k})}_E \, : \, H^* (\mathrm{Syl}_3 (\mathfrak{A}_{3k}, \: \ZZ/3\ZZ) \to H^* (E, \: \ZZ/p\ZZ ) \right) \cap H^* (E, \: \ZZ/3\ZZ )^{W_{\mathfrak{A}_{3k}}(E)}\, .
\end{gather*}
\item
By Theorem \ref{tSteenrod} and induction
\[
\mathrm{res}^{\mathfrak{A}_{3k}}_{E} \, : \, H^* (\mathrm{A}_{3k}, \: \ZZ/3\ZZ ) \to H^* (E, \: \ZZ /3\ZZ )^{W_{\mathfrak{A}_{3k}}(E)}
\]
is surjective (compare also the argument in \cite{Mui}, Prop. 3.9 and Lemma 3.11). 
\end{itemize}

\

Thus 
\[
\mathrm{res}^{\mathfrak{A}_{3k}}_{E} \, : \, H^*_{\mathrm{s}} (\mathrm{A}_{3k}, \: \ZZ/3\ZZ ) \simeq H^*_{\mathrm{s}} (E, \: \ZZ /3\ZZ )^{W_{\mathfrak{A}_{3k}}(E)}
\]
and by the description of the action of $W_{\mathfrak{A}_{3k}}$ on the stable cohomology of $E$, we find that only $e_1\wedge \dots \wedge e_k$ remains spanning the positive dimensional invariants.

\

Finally, to prove the surjectivity of the arrow in the displayed formula \ref{formula2} above, consider the inclusions $E\subset \mathfrak{A}_{3n} \subset \mathfrak{A}_{3n+1}$. Then $\mathrm{Syl}_3 (\mathfrak{A}_{3k}) = \mathrm{Syl}_3 (\mathfrak{A}_{3k+1})$ and, in exact analogy to the argument above, by \cite{Mui}, Prop. 2.2, $E\subset \mathrm{Syl}_3 (\mathfrak{A}_{3k+1}) \subset \mathfrak{A}_{3n+1}$ is a closed system, so that by Card\'{e}nas-Kuhn the image of the cohomology of $\mathfrak{A}_{3n+1}$ in the cohomology of $E$ coincides with the image of the cohomology of $\mathfrak{A}_{3n}$ in $E$ (because also $W_{\mathfrak{A}_{3n}} (E)\simeq W_{\mathfrak{A}_{3n+1}}(E)$).
\end{proof}

\end{document}